\documentclass[12pt,a4paper]{article}
\usepackage{ae,aecompl}
\usepackage[english]{babel}

\usepackage[latin1]{inputenc}
\usepackage{amsfonts,amsmath}
\usepackage{amsthm}
\usepackage{amssymb}
\usepackage{mathrsfs}
\usepackage{graphicx,enumerate}
\usepackage{xy}
\usepackage{./slashbox}

\usepackage{enumitem}
\setlist{nolistsep}

\input xy
\xyoption{all}

\numberwithin{equation}{section}

\newtheorem{thm}{Theorem}[section]

\newtheorem{lemma}[thm]{Lemma}
\newtheorem{prop}[thm]{Proposition}
\newtheorem{defin}[thm]{Definition}

\theoremstyle{remark}
\newtheorem{rmk}[thm]{Remark}

\theoremstyle{remark}

\numberwithin{thm}{section}

\newcommand{\bt}{\textbf}
\newcommand{\de}{\partial}

\newcommand{\wt}{\widetilde}

\newcommand{\Z}{\mathbb{Z}}
\newcommand{\R}{\mathbb{R}}
\newcommand{\C}{\mathbb{C}}

\newcommand{\T}{\mathbb{T}}
\newcommand{\Tb}{\mathbb{T}}
\newcommand{\A}{\mathcal{A}}

\newcommand{\B}{\mathcal{B}}

\newcommand{\eps}{\varepsilon}

\renewcommand{\H}{\mathcal{H}}

\newcommand{\id}{\mathrm{id}}

\renewcommand{\t}{\mathfrak{t}}

\advance\textwidth by 1.4in
\advance\oddsidemargin by -0.6in
\advance\evensidemargin by -0.8in

\title{\vspace{-1cm}Dirac operators\\
 on noncommutative principal torus bundles}

\date{}

\author{Alessandro Zucca and Ludwik D\k {a}browski${}^{*}$\\ %%
\vbox{
\small
\begin{center}
SISSA (Scuola Internazionale Superiore di Studi Avanzati), \\
via Bonomea 265, 34136 Trieste, Italy \\ \ \\
${}^{*}$ Partially supported by PRIN 2010-11 "Operator Algebras, Noncommutative Geometry and Applications".\\
\end{center}
}
}
\begin{document}

\maketitle

\begin{abstract}
Spectral triples over noncommutative principal $\T^n$-bundles are studied, extending recent results about the noncommutative geometry of principal $U(1)$-bundles. We relate the noncommutative geometry of the total space of the bundle with the geometry of the base space. Moreover, strong connections are used to build new Dirac operators. 
We discuss as a particular case the commutative case, the noncommutative tori 
and theta deformed manifolds. 
\end{abstract}

\thispagestyle{empty}

\section{Introduction}
In \cite{AB98,A98} B. Ammann and C. B\"ar discussed the properties of the Riemannian spin geometry of a smooth principal $U(1)$-bundle, relating, under suitable hypotheses, the spin structure and the spinor Dirac operator on the total space of the bundle to the spin structure and the Dirac operator on the base space. A noncommutative generalization of these results was proposed and developed in \cite{DS10,DSZ13} by introducing a notion of projectability for $U(1)$-equivariant \cite{S03} real spectral triples \cite{C94,Bondia} over a noncommutative principal $U(1)$-bundle and showing how it is possible to twist the Dirac operator of a projectable spectral triple using a strong connection. Moreover, a notion of compatibility between a connection and the Dirac operator on the total space of the bundle was introduced.
To the best of our knowledge, Dirac operators on principal bundles with higher dimension structure groups have not been worked out yet. 
In this work we extend the study of $U(1)$-bundles  
to the case of principal $\T^n$-bundles ($\T^n=U(1)^{\times n}$),
working from the beginning with spectral triples and on the noncommutative level. 
In particular we introduce certain twisted Dirac operators, finding in this way a class of new Dirac operators.
We discuss as a particular case the commutative case, the noncommutative tori 
(with the noncommutative 3-torus, seen as a principal $\T^2$-bundle over the circle, 
as an explicit example), and theta deformed manifolds.

\section{Noncommutative principal $\T^n$-bundles}
In noncommutative geometry the algebra of functions (of certain regularity) over a topological space is replaced by a suitable noncommutative algebra.
%, which plays the role of algebra of functions  over noncommutative spaces. 
In an analogue fashion, a Hopf algebra $H$ 
%\cite{Sw69} are 
is usually considered as noncommutative counterpart of a group. 
Then a noncommutative principal bundle should be a particular $H$-comodule algebra,
supplemented by additional structure \cite{BM93,Dur93,Hajac96,BM98,DGH99}; 
here we shall consider the recently elaborated notion of principal comodule algebra \cite{HKMZ07,BZ11}.

\begin{defin}\label{defprinccomodalg}
 Let $H$ be a Hopf algebra with invertible antipode $S$ and counit $\eps$. Then a right $H$-comodule algebra $\A$, with multiplication map $m : \A\otimes\A\rightarrow\A$ and coaction $\Delta_R : \A\rightarrow \A\otimes H$, is called a \emph{principal $H$-comodule algebra} 
if it admits a linear map $\ell : H\rightarrow \A\otimes \A$ such that:
 \begin{description}
 \item[(i)] $\ell(1) = 1\otimes 1$,
 \item[(ii)] $m\circ\ell = \eps$,
 \item[(iii)] $(\ell\otimes\id)\circ\Delta = (\id\otimes\Delta_R)\circ\ell$,
 \item[(iv)] $(S\otimes\ell)\circ\Delta = (\sigma_{\A\otimes H}\otimes\id)\circ(\Delta_R\otimes\id)\circ\ell$,
\end{description}
where $\sigma_{\A\otimes H} : \A\otimes H\rightarrow H\otimes \A$ 
is the switch $\sigma_{\A\otimes H}(a\otimes h) = h\otimes a$.
\end{defin}
%Now let $ denote  of $\A$. 
We shall understand a principal $H$-comodule algebra $\A$ as a quantum principal bundle over 
the coinvariant subalgebra $\B:= \{a\in\A\;|\;\Delta_R(a) = a\otimes 1\}$ of $\A$,
and often denote it simply by the inclusion map $\B \hookrightarrow \A$. 
In particular one can prove that 
%if $A$ is a principal $\H$-comodule algebra 
then $\A$ is a \emph{Hopf-Galois extension} \cite{Hajac96} of $\B$.
This means that is the so-called \emph{canonical map} 
$T_R : \A\otimes_\B \A\rightarrow \A\otimes H$,
$$T_R(a\otimes b) = ab_{(0)}\otimes b_{(1)}$$
is a bijection 
(here and later on we use Sweedler's notation $\Delta_R(a) = a_{(0)}\otimes a_{(1)} \in \A\otimes H$ for the right coaction of $H$ on $\A$).
%; we shall denote it simply by $\B\hookrightarrow\A$.

Moreover, the map $\ell$ 
determines a {\em strong connection form} (for the universal calculus on $\A$)
by setting:
$$\omega : H\rightarrow \Omega^1\A, \quad\quad\omega(h) = \ell(h) - \eps(h).$$
It will be adapted in section \ref{4.1} to the case of the Dirac-operator induced 
differential calculus, and represented as an operator on a Hilbert space.

In this paper $H$ will be the Hopf algebra $H(\T^n)$ of the Lie group $\T^n= U(1)^n$.
It is the polynomial complex unital $^*$-algebra generated by $n$ commuting unitaries 
$z_1,\ldots,z_n$, together with the algebra maps $\Delta$, $S$, $\eps$ 
(coproduct, antipode and counit) defined by:
 $$\Delta(z_i) = z_i\otimes z_i,\quad\quad S(z_i) = z_i^* = z_i^{-1},\quad\quad \eps(z_i) = 1.$$
We introduce the following notation: for $k\in\Z^n$ we set 
$\displaystyle z^k = \prod_{i=1}^n z_i^{k_i}$. 

Consider now a principal comodule algebra $\A$. Let 
$$\A = \bigoplus_{k\in\Z^n} \A^{(k)}$$
be a direct sum decomposition into homogeneous subspaces of the coaction
$$a\in \A^{(k)}\quad\Leftrightarrow\quad \Delta(a) = a\otimes z^k.$$
This corresponds to 
the action of $\T^n$ and of its Lie algebra $\t_n$ on $\A$
$$\delta_j(a) = k_ja, \quad\; \forall a\in \A^{(k)},$$
where $\delta_1,\ldots,\delta_n$ are the usual generators of $\t_n$.

Remark. Actually remaining on the level of (Hopf) algebras we would need here 
only an action of some Hopf algebra $U$ with a nondegenerate dual pairing with $H$. 
This is clearly fulfilled in the case at hand 
by  $U$ given by the universal enveloping algebra of  $\t_n$.
In the following sections we shall assume however that all the relevant structures
discussed on the algebraic level extend to suitable completions,
e.g the coaction of $H$ on $\A$ comes from the strongly continuous unitary action of $T^n$. 
One reason that we keep the coactions of  $H$ besides
the actions of $\T^n$ is for the property of their freenes on the topological ($C^*$-algebra) level,
see \cite{Wor87,BDH13}.

\section{Spectral triples over quantum principal $\T^n$-bundles}\label{3}
%Let now $\A$ be a principal $\O(\T^n)$-comodule algebra and 
Consider now a real spectral triple \cite{C94,Bondia} $(\A,\H,D,J,\gamma)$ 
over the algebra $\A$ (if the triple is odd we set $\gamma = \id$). We denote by $\pi$ the representation of $\A$ on $\H$, but we shall often omit it, writing simply $a\xi$ for $\pi(a)\xi$. 
Recall that then the Dirac operator $D$ has bounded commutators $[D, a]$ 
with any $a\in\A$, and $D^{-1}$ is a compact operator 
on the orthogonal complement of kernel of $D$.
The adjoint action of $J$ maps $\A$ into its commutant and moreover
the following commutation relations between $J$, $D$ and, in the even case, $\gamma$ 
are satisfied
$$J^2 = \epsilon, \quad JD = \epsilon'DJ, \quad J\gamma = \epsilon''\gamma J,$$
depend on the so-called $KR$-dimension of the triple \cite{C94,Bondia},
\begin{equation}\label{tableKRdim}
 \begin{array}{|c|c|c|c|c|}
  \hline
  KR\mathrm{-dim} & 0 & 2 & 4 & 6 \\ \hline
  \epsilon & + & - & - & + \\ \hline
  \epsilon' & + & + &+ & + \\ \hline
  \epsilon'' & + & - & + & - \\ \hline
 \end{array}\quad\quad
\begin{array}{|c|c|c|c|c|}
  \hline
  \mathrm{KR-dim} & 1 & 3 & 5 & 7 \\ \hline
  \epsilon & + & - & - & + \\ \hline
  \epsilon' & - & + & - & + \\ \hline
 \end{array}
\, .
\end{equation}
We shall require moreover that 
%, any spectral triple to fulfil 
the \emph{first order condition} is satisfied; that is, for any $a,b\in\A$,
$$[[D,a],Jb^*J^{-1}] = 0.$$
We recall here that the real structure $J$ induces a right action of $\A$ on $\H$, which commutes with the ordinary representation of $\A$: $\xi\cdot a = Ja^*J^{-1}\xi$, for $\xi\in\H$, $a\in\A$. This action can also be seen as a left action (hence, a representation) of the opposite algebra $\A^\circ$. Furthermore, the Dirac operator $D$ allows to represent differential 1-forms over $\A$ as (bounded) operators on $\H$,
$$\pi_D\bigg(\sum adb \bigg) = \sum a[D,b],$$
and the real structure $J$ induces also a right action of differential forms: $\xi\cdot adb = J\pi_D(adb)^*J^{-1}$. Moreover $\pi_D$ can be used to define a bimodule of differential forms over $\A$. Indeed, we can set
$$\Omega^1_D(\A) = \bigg\{\sum a[D,b]\; \bigg|\; a,b\in\A\bigg\}.$$
We shall assume that the noncommutative spin geometry on $\A$ associated to this spectral triple is invariant with respect to the action of $\T^n$. This can be translated into a condition of equivariance of the spectral triple \cite{S03}.
\begin{defin}
 A \emph{$\T^n$-equivariant real spectral triple} over the algebra $\A$ is a real spectral triple $(\A,\H,D,J,\gamma)$ ($\gamma=\id$ if the triple is odd) together with commuting selfadjoint operators $\delta_j:\H\rightarrow\H$, for $j=1,\ldots,n$, with (common) domain of selfadjointness stable under the action of $\A$, which extend the operators $\delta_j:\A\rightarrow\A$,
$$\delta_j(\pi(a)\psi) = \pi(\delta_j(a))\psi + \pi(a)\delta_j(\psi),$$
and such that
$$\delta_jJ + J\delta_j = 0,\quad\quad [\delta_j,\gamma] = 0,\quad\quad [\delta_j,D] = 0.$$
\end{defin}
\begin{rmk}
 We shall assume that the spectrum of each $\delta_j$ is equal to $\Z$. 
The geometric meaning of this assumption is 
that the action of $\T^n$ on the total space of the bundle lifts to an action 
and not to a projective action on the spinor bundle.
\end{rmk}

Now, if $(\A,\H,D,J,\gamma,\{\delta_j\})$ is a $\T^n$-equivariant real spectral triple, 
the Hilbert space $\H$ splits according to the spectrum of the operators $\delta_j$,
$$\H = \bigoplus_{k\in\Z^n}\H_k,$$
and this decomposition is preserved by the Dirac operator $D$. Moreover, for any $k,l\in\Z^n$,
$\pi(\A^{(k)})\H_l\subseteq\H_{k+l}$. In particular $\H_0$ is stable under the action of the invariant 
subalgebra $\B = \A^{coH} = \A^{(0)}$.

Now we can introduce a notion of projectability for $\T^n$-equivariant spectral triples.
We have to treat separately the odd dimensional and the even dimensional case. 
We begin with the former.

\subsection{Projectable spectral triples: odd case}
\begin{defin}\label{defTnprojodd}
An odd $\T^n$-equivariant real spectral triple $(\A,\H,D,J,\{\delta_j\})$, of $KR$-dimension $n+m$, 
is said to be \emph{projectable} along the fibres if there exists a $\Z_2$ grading $\Gamma$ on $\H$,
which satisfies the following conditions,
$$\Gamma^2 = 1,\quad\quad \Gamma^* = \Gamma,$$
$$[\Gamma,\pi(a)]=0 \quad\forall a\in\A,$$
$$[\Gamma,\delta_j] = 0 \quad\;\mathrm{for}\; j=1,\ldots,n,$$
\begin{equation*}
 J\Gamma = \left\{\begin{array}{lcl}
  \Gamma J & \quad & \mathrm{if} \;\; m\equiv 0\; (\mathrm{mod}\; 4), \\
  -\Gamma J & \quad & \mathrm{otherwise}.
 \end{array}\right.
\end{equation*}
We define the \emph{horizontal} Dirac operator $D_h$ by:
\begin{equation}\label{Dhodd}
D_h = \left\{\begin{array}{lcl}
\displaystyle {1\over 2}\Gamma[D,\Gamma]_- & \quad & \mathrm{for}\; n\;\mathrm{odd} \\ \\
\displaystyle {1\over 2}\Gamma[D,\Gamma]_+ & \quad & \mathrm{for}\; n\;\mathrm{even}
\end{array}\right.
\end{equation}
where $[a,b]_\pm = ab \pm ba$.
\end{defin}
Now we want to impose a condition corresponding to the constant length fibres condition introduced in the $n = 1$ case \cite{DS10}. We give the following definition:
\begin{defin}\label{fiberscond}
We say the bundle $\A$ to have \emph{isometric} fibres if there exists an operator $D_v : \H\rightarrow \H$ such that $D = D_v + D_h + Z$ and:
\begin{description}
\item[(a)] $D_v|_{\H_0} = 0$, where $\H_0$ is the common 0-eigenspace of the derivations $\delta_i$; 
\item[(b)] $[D_v,\Gamma] = 0$ if $n$ is odd, $[D_v,\Gamma]_+ = 0$ if $n$ is even;
\item[(c)] $[D_v,\delta_i] = 0$ for any $i=1,\ldots,n$;
\item[(d)] $Z$ is a bounded operator;
\item[(e)] $Z$ commutes with the elements from $\A$;
\item[(f)] there exists a bounded selfadjoint operator $Z'$ such that $Ja^*J^{-1}(Z\psi) = Z'(Ja^*J^{-1}\psi)$ for any $a\in\A$ and any $\psi\in\H$.
\end{description}
\end{defin} 
Notice that condition (e) has the following consequence: the horizontal Dirac operator $D_h$ and the Dirac operator $D$ determine the same first order differential calculus on $\B$; that is, $[D_h,b] = [D,b]$ for any $b\in\B$. Now we can prove the following results.
\begin{prop}\label{propTnodd}
Let $(\A,\H,D,J,\{\delta_j\},\Gamma)$ be an odd dimensional projectable spectral triple with isometric fibres and let $\H_0$ be the common 0-eigenspace of the derivations $\delta_j$. Then, if we denote by $D_0$ the restriction of $D_h$ to $\H_0$, $(\B,\H_0,D_0)$ is a (usually reducible) spectral triple.

Moreover, if we denote by $J_0$ the restriction of $J$ to $\H_0$, then $J_0$ determines a right action of $\B$ (or a left action of the opposite algebra $\B^\circ$) on $\H_0$ by 
$$hb = b^\circ h = J_0b^*J_0^{-1}h$$
for any $b\in\B$, $h\in\H$. This action fulfils the following properties:
\begin{description}
\item[(a)] $[b,J_0c^*J_0^{-1}] = 0$ for all $b,c\in\B$; that is, $J_0$ maps $\B$ into its commutant;
\item[(b)] $[[D_0,b],J_0c^*J_0^{-1}] = 0$ for all $b,c\in\B$ (\emph{first order condition}).
\end{description}
\end{prop}
\begin{proof}
Clearly $D_h$ is a selfadjoint operator, and it has compact resolvent (see \cite{DS10}). Also, $\B$ preserves $\H_0$ since it is exactly the invariant subalgebra for the $\T^n$-action. Thus $(\B,\H_0,D_0)$ is a spectral triple. We have to prove (a) and (b). (a) follows simply by the fact that $J_0$ is nothing else than $J$, and $J$ maps $\A$, and hence $\B$, into its commutant. For what concerns (b), we recall that the triple over $\A$ satisfies the first order condition; that is,
$$[[D,a],Jb^*J^{-1}] = 0\quad\forall a,b\in\A.$$
Using this fact we can see that:
\begin{equation*}
\begin{split}
[[D_0,b],J_0c^*J_0^{-1}] & = {1\over 2}[[\Gamma D\Gamma,b] \pm [D,b],Jc^*J^{-1}] \\
& = {1\over 2}[\Gamma [D,b]\Gamma,Jc^*J^{-1}] = {1\over 2}\Gamma[[D,b],Jc^*J^{-1}]\Gamma = 0,
\end{split}
\end{equation*}
where we used also the fact that $J\Gamma = -\Gamma J$, according to definition \ref{defTnprojodd}. So the first order condition (b) is fulfilled.
\end{proof}
\begin{lemma}\label{lemmaGammaodd}
Let $(\A,\H,D,J,\{\delta_j\},\Gamma)$, $(\B,D_0,\H_0)$ as above. Then, if we denote by $\gamma_0$ the restriction of $\Gamma$ to $\H_0$,
\begin{equation*}
\begin{array}{lcl}
D_0\gamma_0 = -\gamma_0D_0 & \quad\quad & \mathrm{for}\; n \; \mathrm{odd},\\ \\
D_0\gamma_0 = \gamma_0D_0 & \quad\quad & \mathrm{for}\; n \; \mathrm{even}.
\end{array}
\end{equation*}
\end{lemma}
\begin{proof}
It follows by direct computation, using the definition of $D_h$.
\end{proof}

\subsection{Projectable spectral triples: even case}
\begin{defin}\label{defTnprojeven}
An even dimensional $\T^n$-equivariant real spectral triple $(\A,\H,D,J,\gamma,\{\delta_j\})$, of $KR$-dimension $n+m$, is said to be \emph{projectable} along the fibres if there exists a $\Z_2$ grading $\Gamma$ on $\H$ which satisfies the following conditions,
$$\Gamma^2 = 1,\quad\quad \Gamma^* = \Gamma,$$
$$[\Gamma,\pi(a)]=0 \quad\forall a\in\A,$$
$$[\Gamma,\delta_j] = 0 \quad\;\mathrm{for}\; j=1,\ldots,n,$$
$$J\Gamma = -\Gamma J,$$
$$\Gamma \gamma = (-1)^n\gamma\Gamma.$$
 We define the \emph{horizontal} Dirac operator $D_h$ by:
\begin{equation}\label{Dheven}
D_h = \left\{\begin{array}{lcl}
\displaystyle {1\over 2}\Gamma[D,\Gamma]_- & \quad & \mathrm{for}\; n\;\mathrm{odd} \\ \\
\displaystyle {1\over 2}\Gamma[D,\Gamma]_+ & \quad & \mathrm{for}\; n\;\mathrm{even}
\end{array}\right.
\end{equation}
where $[a,b]_\pm = ab \pm ba$.
\end{defin}

Also in this case we can introduce the isometric fibres condition, see definition \ref{fiberscond}, and prove the analogue of proposition \ref{propTnodd}:
\begin{prop}\label{propTneven}
Let $(\A,\H,D,J,\gamma,\{\delta_j\},\Gamma)$ be an even dimensional projectable spectral triple with isometric fibres and let $\H_0$ be the common 0-eigenspace of the derivations $\delta_j$. Then, if we denote by $D_0$ the restriction of $D_h$ to $\H_0$, $(\B,\H_0,D_0)$ is a (usually reducible) spectral triple.

If we denote by $J_0$ the restriction of $J$ to $\H_0$, then $J_0$ determines a right action of $\B$ (or a left action of the opposite algebra $\B^\circ$) on $\H_0$ by 
$$hb = b^\circ h = J_0b^*J_0^{-1}h$$
for any $b\in\B$, $h\in\H$. And such an action fulfils the following properties:
\begin{description}
\item[(a)] $[b,J_0c^*J_0^{-1}] = 0$ for all $b,c\in\B$; that is, $J_0$ maps $\B$ into its commutant;
\item[(b)] $[[D_0,b],J_0c^*J_0^{-1}] = 0$ for all $b,c\in\B$ (\emph{first order condition}).
\end{description}
Moreover both the operators $\Gamma$ and $\gamma\Gamma$ restricts to $\H_0$, and $\gamma$ anticommutes with $D_0$.
\end{prop}
\begin{proof}
 The proof is the same as that of proposition \ref{propTnodd}.
\end{proof}

\subsection{Real structure and real spectral triples}\label{3.3}
The construction of a real structure for the triples considered in the previous sections requires to discuss separately 4 cases. Indeed, if we denote by $m$ the $KR$-dimension of the triple over $\A$, and we set $j = m - n$ (so that $j$ should be the dimension of the triple over $\B$) we have four different situations: $j$ even and $n$ even, $j$ even and $n$ odd, $j$ odd and $n$ even, 
$j$ odd and $n$ odd.

Before beginning the discussion, we recall here the dependence on the $KR$-dimension of the commutation relations between the real structure, the Dirac operator and the $\Z_2$-grading. We use Connes' selection\footnote{There is another possible choice, see \cite{DD10}.} (see \cite{Bondia,DD10}). Given a real spectral triple $(\A,\H,J,D,\gamma)$ we say that it is of $KR$-dimension $j$ (which we consider always modulo 8) if:
$$J^2 = \eps\cdot\id,$$
$$JD = \eps' DJ,$$
and, for $j$ even,
$$J\gamma = \eps''\gamma J,$$
$$\gamma D = - D\gamma,$$
where $\eps,\eps',\eps'' = \pm 1$ according to table \eqref{tableKRdim}.

\bt{$j$ even, $n$ even.} $(\A,\H,D,J,\gamma)$ is an even real spectral triple of KR-dimension $m = j + n$. We extend the triple $(\B,\H_0,D_0)$ to an even dimensional real spectral triple $(\B,\H_0,D_0,j_0,\gamma_0)$ of KR-dimension $j$, where $j_0$ and $\gamma_0$ are defined in the tables below (the restriction of the operators to $\H_0$ is always understood). We recall that $D_0$ is the restriction of $D_h$ to $\H_0$, where $D_h = {1\over 2}\Gamma[D,\Gamma]_+$, so that $\Gamma D_0 = D_0\Gamma$. Also, we recall that, since $n$ is even, $\Gamma\gamma = \gamma\Gamma$.

\begin{table}[!ht]\caption{$j_0$ and $\gamma_0$ for the even-even case}\label{evenevenJgamma}\begin{center}
\begin{tabular}{|c||c|c|c|c|}
 \hline
 \backslashbox{j}{n} & $0$ & $2$ & $4$ & $6$ \\
 \hline
 \hline
 $0$ & $J$ & $\Gamma J$ & $\Gamma J$ & $J$ \\
 \hline
 $2$ & $J$ & $J$ & $\Gamma J$ & $\Gamma J$ \\
 \hline
 $4$ & $J$ & $\Gamma J$ & $\Gamma J$ & $J$ \\
 \hline
 $6$ & $J$ & $J$ & $\Gamma J$ & $\Gamma J$ \\
 \hline
\end{tabular}$\quad\quad\quad$
\begin{tabular}{|c||c|c|c|c|}
 \hline
 \backslashbox{j}{n} & $0$ & $2$ & $4$ & $6$ \\
 \hline
 \hline
 $0$ & $\gamma$ & $\gamma\Gamma$ & $\gamma$ & $\gamma\Gamma$ \\
 \hline
 $2$ & $\gamma$ & $\gamma\Gamma$ & $\gamma$ & $\gamma\Gamma$ \\
 \hline
 $4$ & $\gamma$ & $\gamma\Gamma$ & $\gamma$ & $\gamma\Gamma$ \\
 \hline
 $6$ & $\gamma$ & $\gamma\Gamma$ & $\gamma$ & $\gamma\Gamma$ \\
 \hline
\end{tabular}
\end{center}
\end{table}

\bt{$j$ even, $n$ odd.} $(\A,\H,D,J)$ is an odd real spectral triple of KR-dimension $m = j + n$. 
We turn the triple $(\B,\H_0,D_0)$ into an even dimensional real spectral triple 
$(\B,\H_0,D'_0,j_0,\gamma_0)$ of KR-dimension $j$, where $\gamma_0 = \Gamma|_{\H_0}$ and
 $j_0$, $D'_0$ are defined in the tables below\footnote{In the cases with $(j,n)$ equal to $(0,3)$,
 $(0,5)$, $(4,3)$ and $(4,5)$, actually, the real structure $j_0$ does not fulfil the right commutation
 relations. Indeed, $j_0^2$ has the wrong sign. For a discussion of this issue see remark
 \ref{rmkoddeven}.} (the restriction of the operators to $\H_0$ is always understood). 
We recall that $D_0$ is the restriction of $D_h$ to $\H_0$, where 
$D_h = {1\over 2}\Gamma[D,\Gamma]$, so that $\Gamma D_0 = - D_0\Gamma$.

\begin{table}[!ht]\caption{$D'_0$ and $j_0$ for the even-odd case}\label{evenoddDJ}\begin{center}
\begin{tabular}{|c||c|c|c|c|}
 \hline
 \backslashbox{j}{n} & $1$ & $3$ & $5$ & $7$ \\
 \hline
 \hline
 $0$ & $D_0$ & $D_0$ & $D_0$ & $D_0$ \\
 \hline
 $2$ & $D_0$ & $\Gamma D_0$ & $\Gamma D_0$ & $D_0$ \\
 \hline
 $4$ & $D_0$ & $D_0$ & $D_0$ & $D_0$ \\
 \hline
 $6$ & $D_0$ & $\Gamma D_0$ & $\Gamma D_0$ & $D_0$ \\
 \hline
\end{tabular}$\quad\quad\quad$
%\end{center}
%\end{table}
%\begin{table}[!ht]\caption{$j_0$ for the even-odd case}\label{evenoddDJ}\begin{center}
\begin{tabular}{|c||c|c|c|c|}
 \hline
 \backslashbox{j}{n} & $1$ & $3$ & $5$ & $7$ \\
 \hline
 \hline
 $0$ & $\Gamma J$ & $J\footnotemark[\value{footnote}]$ & $\Gamma J\footnotemark[\value{footnote}]$ & $J$ \\
 \hline
 $2$ & $J$ & $J$ & $\Gamma J$ & $\Gamma J$ \\
 \hline
 $4$ & $\Gamma J$ & $J\footnotemark[\value{footnote}]$ & $\Gamma J\footnotemark[\value{footnote}]$ & $J$ \\
 \hline
 $6$ & $J$ & $J$ & $\Gamma J$ & $\Gamma J$ \\
 \hline
\end{tabular}
\end{center}
\end{table}

\begin{rmk}\label{rmkoddeven}
 We have to spend some words about the cases with $(j,n)$ equal to $(0,3)$, $(0,5)$, $(4,3)$ 
and $(4,5)$. In all these situations, indeed, it is not possible to find a set of operators 
$(D_0,j_0,\gamma_0)$ constructed only using $\Gamma$, $D$ and $J$ and fulfilling all the required
commutation relations. And this issue can not be solved changing the commutation relation between
$J$ and $\Gamma$. Indeed, there are two possible choices: $J\Gamma = \Gamma J$ 
and $J\Gamma = -\Gamma J$. In the first case $\gamma_0 = \Gamma$ fulfils all the required 
commutation relations, but $j_0^2$ has the wrong sign (that is, $j_0^2 = -1$ for $j\equiv 0\; 
(\mathrm{mod}\; 8)$ and $j_0^2 = 1$ for $j\equiv 4\; (\mathrm{mod}\; 8)$). In the second one, 
instead, it is possible to recover a $j_0$ with the correct commutation relations, by setting 
$j_0 = \Gamma J$, but then we can not find a suitable $\gamma_0$ commuting with $j_0$. 
We have chosen to adopt the first convention, since it allows to define all the three operators, 
even if with $j_0^2$ with the wrong sign, and, moreover, it appears as the more natural choice 
\cite{ZPhD}. We conclude this remark with the following observation: the fact that we are not able 
to define a $j_0$ fulfilling all the right commutation relations does not mean that such a $j_0$ 
does not exist, but only that it can not be expressed only in terms of $J$, $D$ and $\Gamma$.
$\diamond$
\end{rmk}

\bt{$j$ odd, $n$ even.} $(\A,\H,D,J)$ is an odd real spectral triple of KR-dimension $m = j + n$. 
We turn the triple $(\B,\H_0,D_0)$ into an odd dimensional real spectral triple $(\B,\H_0,D'_0,j_0)$ 
of KR-dimension $j$, where $j_0$ and $D'_0$ are defined in the tables below (the restriction 
of the operators to $\H_0$ is always understood). We recall that $D_0$ is the restriction of $D_h$ 
to $\H_0$, where $\displaystyle D_h = {1\over 2}\Gamma[D,\Gamma]_+$, so that 
$\Gamma D_0 = D_0\Gamma$.

\begin{table}[!ht]\caption{$D'_0$ and $j_0$ for the odd-even case}\label{oddevenDJ}\begin{center}
\begin{tabular}{|c||c|c|c|c|}
 \hline
 \backslashbox{j}{n} & $0$ & $2$ & $4$ & $6$ \\
 \hline
 \hline
 $1$ & $D_0$ & $\Gamma D_0$ & $D_0$ & $\Gamma D_0$ \\
 \hline
 $3$ & $D_0$ & $\Gamma D_0$ & $D_0$ & $\Gamma D_0$ \\
 \hline
 $5$ & $D_0$ & $\Gamma D_0$ & $D_0$ & $\Gamma D_0$ \\
 \hline
 $7$ & $D_0$ & $\Gamma D_0$ & $D_0$ & $\Gamma D_0$ \\
 \hline
\end{tabular}$\quad\quad\quad$
%\end{center}
%\end{table}
%\begin{table}[htbp]\caption{$j_0$ for the odd-even case}\label{oddevenDJ}\begin{center}
\begin{tabular}{|c||c|c|c|c|}
 \hline
 \backslashbox{j}{n} & $0$ & $2$ & $4$ & $6$ \\
 \hline
 \hline
 $1$ & $J$ & $\Gamma J$ & $\Gamma J$ & $J$ \\
 \hline
 $3$ & $J$ & $J$ & $\Gamma J$ & $\Gamma J$ \\
 \hline
 $5$ & $J$ & $\Gamma J$ & $\Gamma J$ & $J$ \\
 \hline
 $7$ & $J$ & $J$ & $\Gamma J$ & $\Gamma J$ \\
 \hline
\end{tabular}
\end{center}
\end{table}

\bt{$j$ odd, $n$ odd.} $(\A,\H,D,J,\gamma)$ is an even real spectral triple of KR-dimension $m = j + n$. We turn the triple $(\B,\H_0,D_0)$ into an odd dimensional real spectral triple $(\B,\H_0,D'_0,j_0)$ of KR-dimension $j$, where $j_0$ and $D'_0$ are defined in the tables below (the restriction of the operators to $\H_0$ is always understood). We recall that $D_0$ is the restriction of $D_h$ to $\H_0$, where $\displaystyle D_h = {1\over 2}\Gamma[D,\Gamma]$, so that $\Gamma D_0 = - D_0\Gamma$.

\begin{table}[!ht]\caption{$D'_0$ and $j_0$ for the odd-odd case}\label{oddoddDJ}\begin{center}
\begin{tabular}{|c||c|c|c|c|}
 \hline
 \backslashbox{j}{n} & $1$ & $3$ & $5$ & $7$ \\
 \hline
 \hline
 $1$ & $D_0$ & $D_0$ & $\Gamma D_0$ & $\Gamma D_0$ \\
 \hline
 $3$ & $D_0$ & $D_0$ & $D_0$ & $D_0$ \\
 \hline
 $5$ & $D_0$ & $D_0$ & $\Gamma D_0$ & $\Gamma D_0$ \\
 \hline
 $7$ & $D_0$ & $D_0$ & $D_0$ & $D_0$ \\
 \hline
\end{tabular}$\quad\quad\quad$
\begin{tabular}{|c||c|c|c|c|}
 \hline
 \backslashbox{j}{n} & $1$ & $3$ & $5$ & $7$ \\
 \hline
 \hline
 $1$ & $\Gamma J$ & $\Gamma J$ & $J$ & $J$ \\
 \hline
 $3$ & $J$ & $\Gamma J$ & $\Gamma J$ & $J$ \\
 \hline
 $5$ & $\Gamma J$ & $\Gamma J$ & $J$ & $J$ \\
 \hline
 $7$ & $J$ & $\Gamma J$ & $\Gamma J$ & $J$ \\
 \hline
\end{tabular}
\end{center}
\end{table}

We conclude this section pointing out that, in all the cases discussed above, the real structure $j_0$  maps the algebra $\B$ into its commutant and the triple fulfils the first order condition. Both properties follow from proposition \ref{propTnodd} (and from the analogue result in the even dimensional case, see proposition \ref{propTneven}).

\subsection{$D$-connections and twisted Dirac operators}\label{3.4}

We recall some of the results in \cite{DS10} about a way to twist real spectral triples by a left-module equipped with a hermitian connection.

Let $(\B,\H,D,J)$ be a real spectral triple over a (unital) algebra $\B$. Consider another Hilbert space $\H_M$ together with a representation of $\B$. Let $M$ be the space of $\B$-linear bounded maps $m : \H \rightarrow \H_M$. Assume that:
\begin{description}
\item[(a)] $\H M \equiv M(\H)$ is dense in $\H_M$, where $M(\H)$ is the linear span of elements $m(h)$, $m\in M$, $h\in\H$;
\item[(b)] the multiplication map from $\H\otimes_\B M$ to $\H M$ is an isomorphism.
\end{description}
Then using the right $\B$-module structure induced on $\H$ by the real structure $J$, namely 
\begin{equation}\label{rightaction}
hb \equiv Jb^*J^{-1}h
\end{equation} 
for any $h\in\H$ and any $b\in\B$, one has a left $\B$-module structure on $M$ through:
$$(bm)(h) = m(hb)\quad\;\;\forall m\in M.$$
Writing the action of $M$ on the right, that is $m(h) \equiv hm$, the $\B$-linearity reads
$$(bh)m = b(hm),$$
while the left $\B$ action on $M$ becomes
$$h(bm) = (hb)m.$$
Also, it follows from the order one condition (see 
e.g. \cite{Bondia}) that there is a right action of $\Omega^1_D(\B)$ on $H$, given by:
\begin{equation}\label{homega}
h\omega = - J\omega^*J^{-1}h\quad\;\;\forall \omega\in\Omega^1_D(\B), 
\end{equation}
where $\omega^*$ is the adjoint of $\omega$, s.t. $([D,b])^* = -[D,b^*]$ and
$$h[D,b] = D(hb) - (Dh)b.$$
Such an action is clearly left $\B$-linear. Also, it induces a left action of $\Omega^1_D(\B)$ on $M$ and $\Omega^1_D(\B)M$ is just the space of all compositions $m\circ \omega$ of left $\B$-linear maps. 

Next we adopt from \cite{DS10} suitable connections (covariant derivatives). 
\begin{defin}\label{defDconn}
We call a linear map $\nabla : M\rightarrow \Omega^1_D(\B)M$ a \emph{$D$-connection} on $M$ 
if it satisfies:
$$\nabla(bm) = [D,b]m + b\nabla(m), \;\;\forall b\in\B,\; m\in M.$$
\end{defin}
Since we are dealing with maps between Hilbert spaces, we have 
the adjoint $m^\dag$ of $m\in M$, 
and the adjoint $\eta^\dag$ of a 1-form $\eta\in \Omega^1_D(\B)$
(which coincides with $\eta^*$, where the star operation is extended to forms). 
Thus we can define the adjoint of an element of $\Omega^1_D(\B)M$ simply by $(\eta m)^\dag = m^\dag\eta^\dag$. Of course, it will not be an element of $\Omega^1_D(\B)M$, but we do not need this. Now we can introduce the notion of hermiticity for a $D$-connection.
\begin{defin}\label{hermitDconn}
A $D$-connection $\nabla$ is said to be \emph{hermitian} if, for each $m_1,m_2\in M$,
\begin{description}
\item[(i)] as an operator on $\H$, $m_1^\dag\circ m_2\in J\B J^{-1}$;
\item[(ii)] writing the actions on arbitrary $h\in\H$ on the right, we have:
$$h\nabla(m_2)m_1^\dag - hm_2\nabla(m_1)^\dag = (Dh)m_2m_1^\dag - D(hm_2m_1^\dag).$$
\end{description}
\end{defin}
Given a $D$-connection in \cite{DS10} certain operator $D_M$ is defined over the dense set 
in $\H_M$.
\begin{defin}
Define $D_M$ on $M(\mathrm{Dom}(D))\subset \H_M$ by:
$$D_M(hm) = (Dh)m + h\nabla(m)\quad\;\forall m\in M.$$
\end{defin}
\begin{prop}
If $\nabla$ is a hermitian $D$-connection, the operator $D_M$ is selfadjoint and has bounded commutators with $\B$.
\end{prop}
\begin{proof}
See \cite{DS10}, proposition 4.7.
\end{proof}

\section{Projectable spectral triples and twisted Dirac operators}
Let $(\A,\H,D,J,\gamma,\{\delta_j\},\Gamma)$ be a projectable $\T^n$-equivariant real spectral triple, where $\A$ is a principal $H(\T^n)$-comodule algebra, with invariant subalgebra $\B$. Assume that the triple satisfies the isometric fibres condition. Then the construction discussed in the previous section can be used to build twisted Dirac operators from the horizontal Dirac operator $D_h$. The first thing we need is a suitable notion of connection.

\subsection{Operators of strong connection for the Dirac calculus}\label{4.1}

The Dirac operator of a spectral triple $(\A,\H,D)$ defines a first order differential calculus $\Omega^1_D(\A)$, given by the linear span of all operators of the form $a[D,b]$, $a,b\in\A$, the differential of $a$ being $da = [D,a]$. Generalizing \cite{DS10}, we say that $\Omega^1_D(\A)$ is \emph{compatible} with the de Rham calculus on $\T^n$ iff
$$\sum_j a_j[D,b_j] = 0 \quad \Rightarrow \quad \sum_j a_j \delta_i(b_j) = 0\quad \forall i=1,\ldots,n.$$
This compatibility condition is obtained \cite{ZPhD} by requiring that $\A$, with the calculus $\Omega^1_D(\A)$, is a quantum principal bundle with general calculus \cite{BM93}, compatible with the de Rham calculus on the Hopf algebra $H(\T^n)$. This allows us to give the following definition of strong connection (see \cite{ZPhD} for the relation between this definition and the ordinary notion of strong connection over a quantum principal bundle with general calculus \cite{BM93,Hajac96}).
\begin{defin}\label{sconnfamily}
 A family of $n$ 1-forms $\{\omega_i\}\subset\Omega^1_D(\A)$ is called a \emph{strong $\T^n$-connection} for the $\T^n$-bundle $\A$ if the following conditions hold:
\begin{description}
 \item[(i)] $\delta_j(\omega_i) = 0$ for any $i,j = 1,\ldots,n$;
 \item[(ii)] if $\omega_i = \sum_j p_jdq_j$, with $p_j,q_j\in A$, then $\sum_j p_j\delta_i(q_j) = 1$ and $\sum_j p_j\delta_l(q_j) = 0$ for $l\neq i$;
 \item[(iii)] $\forall a\in\A, \;\; \left(da - \sum_i \delta_i(a)\omega_i\right)\in\Omega^1_D(\B)\A$.
\end{description}
\end{defin}
A strong $\T^n$-connection $\{\omega_i\}$ defines an $\Omega^1_D(\A)$-valued strong connection form, in the sense of \cite{BM93,Hajac96}, by:
$$\omega(z^k) = \sum_{k=1}^n k_i\omega_i.$$

\subsection{Twisted Dirac operators}
We exploit strong connections to build twisted Dirac operators. 
We work with one of the triples $(\B,\H_0,D_0,j_0)$ constructed in sect.\ref{3} 
and we take $M = \A^{(k)}$, since it satisfies both the conditions (a) and (b). 
First we construct a $D_0$-connection on $\A^{(k)}$.
\begin{prop}
Let $\omega$ be a strong $\T^n$-connection defined by a family $\{\omega_i\}_{i=1,\ldots, n}\subseteq\Omega^1_D(\A)$. Then, for any $k\in\Z^n$, the map $\nabla_\omega : \A^{(k)}\rightarrow \Omega^1_D(\A)\A^{(k)}$ defined by
$$\nabla_\omega(a) = [D,a] - \sum_{i=1}^n k_ia\omega_i,$$
where both $a\in\A^{(k)}$ and $\nabla_\omega(a)$ are regarded as operators on $\H_0$ acting from the right, defines a $D_0$-connection over the left $\B$-module $\A^{(k)}$, where $D_0$ denotes the restriction of the horizontal Dirac operator $D_h$ to $\H_0$.
\end{prop}
\begin{proof}
 The proof is the same as that of proposition 5.4 in \cite{DS10}.
\end{proof}
\begin{prop}
The $D_0$-connection $\nabla_\omega$ is hermitian if all the $\omega_i$ are selfadjoint (as operators on $\H$).
\end{prop}
\begin{proof}
We have to check (i) and (ii) of definition \ref{hermitDconn}. Since we have taken $M = \A^{(k)}$ acting on $\H_0$ on the right via $h a = Ja^*J^{-1}h$, and since $J$ maps $\A$ into its commutant, then (i) is fulfilled. For what concerns (ii), we proceed by direct computation: let $a_1,a_2\in\A^{(k)}$ and $h\in\H_0$; then, using \eqref{homega}, we get:
\begin{equation*}
\begin{split}
h&\bigg(\nabla_\omega(a_2)a_1^\dag - a_2\nabla_\omega(a_1)^\dag - (Dh)a_2a_1^\dag + D(ha_2a_1^\dag) \bigg) \\
&= h\bigg([D,a_2] - \sum_{i=1}^n k_ia_2\omega_i\bigg)a_1^\dag - h\bigg(a_2\bigg([D,a_1] - \sum_{i=1}^n k_ia_1\omega_i\bigg)^\dag\bigg) - h[D,a_2a_1^\dag] \\
& = h\bigg(\sum_{i=1}^n k_ia_2(\omega_i^\dag - \omega_i)a_1^\dag\bigg),
\end{split}
\end{equation*}
which vanishes if $\omega_i^\dag = \omega_i$.
\end{proof}

Now, we can identify, up to completion, $\H_0\A^{(k)}$ with $\H_k$; then the construction discussed in the previous section gives us a family of spectral triples $(\B,\H_k,D_\omega^{(k)}$), $k\in\Z^n$, where each $D_\omega^{(k)}$ is the twisted Dirac operator constructed using the connection $\nabla_\omega$ on $\A^{(k)}$. Taking $D_\omega$ to be the closure of the direct sum of the Dirac operators of this family we obtain a twisted Dirac operator $D_\omega$, acting on (a dense domain of) the whole Hilbert space $\H$.
\begin{prop}\label{propDomegaTn}
The twisted Dirac operator $D_\omega$ is selfadjoint if all the $\omega_i$ are selfadjoint one-forms, and it has bounded commutators with all the elements of $\A$.
\end{prop}
\begin{proof}
We compute the action of $D_\omega$ on an element $hp$ in its domain, with $h\in\H_0$ and $p\in\A^{(k)}$ (we use \eqref{homega} for the right action of one-forms, where $J_0$ stands either\footnote{If $D_0'$ - see tables in the previous section - is simply $D_0$ then we take $J_0 = j_0$; if, instead, $D_0' = \Gamma D_0$, we take $J_0 = \Gamma j_0$.} for $j_0$ or $\Gamma j_0$):
\begin{equation}\label{Domega}
\begin{split}
D_\omega(hp) & = (D_0h)p + h[D,p] - \sum_{i=1}^n k_ihp\omega_i \\
& = (D_0h)p + [D,J_0p^*J_0^{-1}]h + \sum_{i=1}^nJ_0\omega_i^*J_0^{-1}k_ihp \\
& = D(hp) + ((D_0 - D)h)p + \sum_{i=1}^nJ_0\omega_i^*J_0^{-1}h\delta_i(p) \\
& = \left(D + \sum_{i=1}^n J_0\omega_i^*J_0^{-1}\delta_i - Z'\right)(hp).
\end{split}
\end{equation}
Now, the Dirac operator $D$ and the derivations $\delta_i$ are selfadjoint, $Z$ and $\omega$ are bounded and selfadjoint; moreover, any $\delta_i$ is relatively bounded with respect to $D$. Then, by Kato-Rellich theorem, $D_\omega$ is selfadjoint on $\H$.

Next, $D$ has bounded commutator with each $a\in\A$ and, since any $\omega_i$ is a one-form, from the first order condition (which holds also for the triple $(\B,\H_0,D_0)$, see proposition \ref{propTnodd}) the commutator of the second term with $a$ is $\sum_i J_0\omega_i^*J_0^{-1}\delta_i(a)$ and hence is bounded. The third term of \eqref{Domega} gives commutators between bounded operators, since $Z$ is bounded, and thus it gives only bounded terms. Therefore $[D_\omega,a]$ is bounded for each $a\in\A$.
\end{proof}
\begin{prop}
Let $D_v$ 
%and  $Z$ 
be as in definition \ref{fiberscond}. Define
$$\mathcal{D}_\omega = D_v + D_\omega.$$
Then $(\A,\H,\mathcal{D}_\omega)$ is a projectable spectral triple with isometric fibres, and the horizontal part of the operator $\mathcal{D}_\omega$ coincides with $D_\omega$.
\end{prop}
\begin{proof}
See proof of proposition 5.8 in \cite{DS10}.
\end{proof}

As in \cite{DS10} we introduce the following notion of \emph{compatibility}.
\begin{defin}\label{defcomp}
We say that a strong connection $\omega$ is \emph{compatible} with a Dirac operator $D$ if 
$D_\omega$ and $D_h$ coincide on a dense subset of $\H$.
\end{defin}

\section{Spin geometry of principal $\T^n$-bundles}\label{6}

In this section we shall show that, under suitable hypotheses, the classical $\T^n$-bundles satisfy
 (as they should) the assumptions, and possess all the structures, discussed above in the 
noncommutative setup. Let $M$ be an $(m+n)$-dimensional oriented compact smooth manifold 
that is the total space of a principal $\T^n$-bundle over the $m$-dimensional oriented manifold 
$N = M/\T^n$. 
Assume that $M$, $N$ are Riemannian manifolds, with metric tensors, respectively, $\wt{g}$ 
and $g$ such that:
\begin{description}
\item[-] the action of $\T^n$ is isometric w.r.t. $\wt{g}$;
\item[-] the bundle projection $\pi : M\rightarrow N$ is an orientation preserving 
Riemannian submersion;
\item[-] the fibres are isometric one to each other; moreover, the length of each fundamental vector
 field $K_a$ is constant along $M$.
\end{description}
Denoting by $\{T_a\}_{a=1,\ldots,n}$ the canonical basis of the Lie algebra of $\T^n$, we assume 
each $K_a$ to be the fundamental vector field associated to $T_a$. The last assumption could be 
weakened (as in the case of $U(1)$ bundles \cite{A98,AB98}) but such a more general situation 
would be more difficult to treat in the noncommutative case, and won't be considered here.

Under these assumptions, there is a unique principal connection 1-form 
$\omega : TM\rightarrow \t_n$ such that $\ker\omega$ 
is orthogonal to the fibres, at any point of $M$, with respect to the metric 
$\wt{g}$. If $\{T_a\}_{a=1,\ldots,n}$ is the canonical basis of the Lie algebra of $\T^n$, 
then $\omega$ will be of the form
$$\omega = \sum_{a=1}^n \omega_a\otimes T_a,$$
where each $\omega_a$ is a $\C$-valued 1-form on $M$. Next, for any vector field $X$ on $N$ 
we shall denote by $\wt{X}$ its horizontal lifting. Consider now a (local) oriented orthonormal frame 
$f = \{f_1,\ldots,f_m\}$ on $N$. Then, if we set 
\begin{equation*}
\left\{\begin{array}{lcl}
e_a =  {1\over l_a}K_a & \quad & a=1,\ldots,n, \\
e_{j+n} = \wt{f_j} & \quad & j=1,\ldots m,
\end{array}\right.
\end{equation*}
where $l_j$ are real positive constants, then $e = \{e_k\}_{k=1,\ldots,n+m}$ is a (local) orthonormal frame on $M$ (oriented with respect to a suitably chosen orientation).

Assume now that $M$ is a spin manifold, and let $\Sigma M$ be its spinor bundle. 
We also assume the $\T^n$ action lifts to an action $\T^n\times\Sigma M\rightarrow\Sigma M$,
in which case we shall speak of \emph{projectable} spin structure. 
A projectable spin structure on $M$ induces a spin structure on $N$ (this is a 
straightforward consequence of the analogue property for the $U(1)$ case \cite{AB98}). 
Then the Dirac operator $\wt{D}$, acting on $L^2$-sections of $\Sigma M$, will be the following one:
$$\wt{D} = \sum_{i=1}^{n+m} \gamma^i\de_{e_i} + {1\over 4}\sum_{i,j,k=1}^{n+m}
 \wt{\Gamma}^k_{ij}\gamma^i\gamma^j\gamma^k,$$
where the $\gamma^j$ are the gamma matrices, associated to the orthonormal frame $\{e_j\}$,
generating the action of the $(n+m)$-dimensional Clifford algebra and $\wt{\Gamma}^k_{ij}$ are the 
Christoffel symbols of the Levi-Civita connection on $TM$ for the frame $\{e_j\}$. Using the
letters $a,b,c...$ to denote indices from 1 to $n$ and the letters $i,j,k...$ to denote indices from
 $n+1$ to $n+m$, we have:
$$\wt{\Gamma}^k_{ij} = \Gamma^k_{ij},$$
$$-\wt{\Gamma}^a_{ij} = \wt{\Gamma}^j_{ia} = \wt{\Gamma}^j_{ai} = {l_a\over 2}d\omega_a(e_i,e_j),$$
$$\wt{\Gamma}^a_{ib} = \wt{\Gamma}^a_{bi} = \wt{\Gamma}^i_{ab} = \wt{\Gamma}^a_{bc} = 0,$$
where the $\Gamma^k_{ij}$ are the Christoffel symbols of the Levi-Civita connection on $TN$, 
with respect to the frame $f$. Before going on, we notice that the Lie derivative with respect 
to each Killing vector field differs from the spinor covariant derivative by:
$$\nabla_{e_a} = \de_{e_a} + {l_a\over 4} \sum_{j<k} d\omega_a(e_j,e_k)\gamma^j\gamma^k.$$
Now we want to express the Dirac operator $\wt{D}$ as a sum of two first order operators plus 
a zero order term. The first operator, which we shall call the \emph{vertical Dirac operator}, 
is given by:
$$D_v = \sum_{a=1}^n {1\over l_a}\gamma^a\de_{K_a} = \sum_{a=1}^n \gamma^a\de_{e_a}.$$
In order to construct the second operator we split the Hilbert space 
$L^2(\Sigma M)$ into irreducible representations of $\T^n$:
$$L^2(\Sigma M) = \bigoplus_{k\in\Z^n} V_k.$$
Next, for simplicity we discuss only the case when both $m$ and $n$ are even, 
the other cases can be obtained in a similar way. 
For any $k\in\Z^n$ consider the
% vector space $\C$ carrying the 
irreducible representation on $\C$ with weight $k$ of $\T^n$ 
and form the associated complex bundle 
$L_k = M\times_{\T^n}\C$. Moreover, endow it with the connection $i\omega$. 
Let $\Sigma_m$ denote the $m$-dimensional spinor representation.
\begin{prop}
For each $k\in\Z^n$ there is an isomorphism
$$Q_k : L^2(\Sigma N\otimes L_k)\otimes \Sigma_m \rightarrow V_k$$
such that the horizontal covariant derivatives, with respect to the vector fields $\wt{f_i}$, are given by
$$\nabla_{\wt{f_i}}Q_k(\psi) = 
Q_k(\nabla_{f_i}\psi) + \sum_{j=n+1}^{n+m}
\sum_{a=1}^n {l_a\over 4}d\omega_a(e_i,e_j)\gamma(K_a/l_a)\gamma(e_j)Q_k(\psi).$$
Moreover, Clifford multiplication is preserved, i.e.
$$Q_k(\gamma(X)\psi) = \gamma(\wt{X})Q_k(\psi).$$
\end{prop}
\begin{proof}
We can write $\Sigma M = SM\times_{\mathrm{Spin}(n+m)}\Sigma_{n+m}$ and 
$\Sigma N = SN\times_{\mathrm{Spin}(m)}\Sigma_m$ where $SM$, $SN$ are, respectively, 
the principal $\mathrm{Spin}(m+n)$- and $\mathrm{Spin}(m)$-bundles defining the spin structures 
of the two manifolds and $\Sigma_{n+m}$, $\Sigma_n$ are the canonical spin representations 
of the spin groups. Then, since we assumed both $m$ and $n$ even, we have: 
$\Sigma_{n+m} = \Sigma_n\otimes\Sigma_m$. 
Then the proposition follows by direct computations, cfr. the proof of lemma 4.4 in \cite{AB98}.
\end{proof}
Then one can see, by direct computation, that, if we define the \emph{horizontal Dirac operator}, 
on each $V_k$, by
$$D_h = Q_k\circ (D\otimes\id)\circ Q_k^{-1},$$
where $D$ is the (twisted) Dirac operator on $\Sigma N\otimes L_k$, then $Z = \wt{D} - D_v - D_h$ 
is a zero order operator, which takes the form
$$Z = -{1\over 4}\sum_{a=1}^n l_a\gamma(K_a/l_a)\gamma(d\omega_a).$$

Now, it is clear that with $\Gamma =  \gamma^1 \gamma^2 \dots \gamma^n$,
where $\gamma^a, a=1,\dots, n$ are the gamma matrices associated to the (vertical part) of 
the orthonormal frame $\{e_j\}$, $(C^\infty(M), L^2(\Sigma), \wt{D})$ is a projectable spectral triple 
in our sense, and $D_h$ comes as above.
Next, with $D_v$ and $Z$ as above the %$(\A,\H,\mathcal{D}_\omega)$ 
fibres are isometric.
Moreover, identifying the usual differential (de Rham) calculus with the Dirac calculus $\Omega^1_{\wt{D}}(\C^\infty(M))$,  we observe that the 
connection  $\omega$ is of course strong 
(as is any connection in the usual sense) and that it is compatible with the Dirac operator $\wt{D}$.
In addition  $\omega$ is also compatible with the the Dirac operator 
$\mathcal{D}_\omega = D_v + D_\omega$
(which however corresponds to a metric connection possibly with nonzero torsion).

\section{Noncommutative tori. Example of $\T^3_\theta$ as quantum principal 
$\T^2$-bundle}\label{sectexT3theta}

Noncommutative tori are most studied examples of noncommutative spaces
(see e.g. \cite{Bondia}, chapter 12) on the  topological, smooth and metric level.
The underlying algebra $\A=\A(\T^k_\theta)$ of the $k$-dimensional noncommutative torus
is generated by $n$ unitaries $U_1,\ldots,U_k$ with the relations
$$U_iU_j = e^{2\pi\theta_{ij}}U_jU_i,$$
where $\theta = (\theta_{ij})$ is an $k\times k$ skewsymmetrix real matrix. 
Of course, for $\theta = 0$, we recover the coordinate algebra of the usual $k$-torus $\T^k$.

There is an action of $\T^k$ on $\A(\T^k_\theta)$ which on the Lie algebra level is 
given by the commuting derivations
\begin{equation}\label{LietnactTntheta}
\delta_i(U_j) = \delta_{ij}U_j.
\end{equation}
where $\delta_1,\ldots,\delta_k$ are standard generators of $\t_k = \mathrm{Lie}(\T^k)$.
\\
Next on the Hilbert space 
$\H=L^2(\T^{k}_\theta,\tau)\otimes \C^{2^{[k/2]}}$,
where $\tau$ is the usual trace on $\A(\T^{k}_\theta)$ and $[\,\, ]$ denote the integer part.
%  of $k$.
The standard Dirac operator is defined by the formula 
$$D=\sum_j^{m+n} \gamma_j \delta_j ,$$
where the $2^{[n/2]}\times 2^{[n/2]}$ gamma matrices are irreducible representations 
of the Clifford algebra $Cl(m+n)$,
and the derivations $\delta_j$ have been implemented on $\H$ (via commutators).

We observe that using (part of) this action any $(n+m)$-dimensional noncommutative torus 
can be viewed as a principal $\T^n$-bundle over an $m$-dimensional noncommutative torus. 
Indeed, the last $m$ generators, $U_1,\ldots,U_n$, of $\T^{n+m}_\theta$ 
generate the noncommutative torus $\T^m_{\theta'}$, where $\theta'$ is the $m\times m$ 
lower--right block submatrix of $\theta$.
The algebra $\A(\T^m_{\theta'})$ is indeed just the invariant subalgebra $\B=(\A(\T^{n+m}_\theta))^{coH}$ of the coaction
$\Delta_R$ of $H = H(\T^n)$ on $\A(\T^{n+m}_\theta)$, defined by
\begin{equation*}
 \begin{array}{lcl}
  \Delta_R(U_j) = U_j\otimes z_j & \quad & j\leq n, \\
  \Delta_R(U_{n+j}) = U_{n+j}\otimes 1 .& & 
 \end{array}
\end{equation*}
Of course it is the same as invariant subalgebra of the action of $\T^n$:
\begin{equation*}
 \begin{array}{lcl}
  \beta: U_j\mapsto z_j U_j & \quad & j\leq n, \\
  \beta: U_{n+j} \mapsto U_{n+j} .& & 
 \end{array}
\end{equation*}

The Dirac differential calculus $\Omega^1_D(\T^{n+m}_\theta)$ coincides with the so called derivative
 based calculus and can be easily seen to be compatible with the de Rham calculus on $\T^n$,
(which in fact is also isomorphic to the Dirac calculus on $\T^n$).

Moreover $\T^{n+m}_\theta$ is a $\T^n$ quantum principal bundle in our sense and we can give 
explicit formula for arbitrary strong connection (for the Dirac calculus).
\begin{prop}\label{propTnthetasconn}
$\omega : H\rightarrow \Omega^1_D(\T^{n+m}_\theta)$ is a strong connection form iff
for any $k\in\Z^n$,
\begin{equation}\label{eqTnthetasconn}
 \omega(z^k) = \sum_{i=1}^n\sum_{j=1}^m k_ib_{ij}\otimes \gamma^j 
+ \sum_{i=1}^n k_i\otimes \gamma^{m+i},
\end{equation}
with $b_{ij}\in\T^m_{\theta'}$.
\end{prop}

Moreover it can be shown that the standard spectral triple (with the 'flat' Dirac operator) on $\A$
is projectable and the horizontal Dirac operator is again the standard Dirac operator on $\B$.
One can discuss also the twisted Dirac operators and the compatibility issue.
This was accomplished for the case of  $U(1)=\T^1$-bundles  in \cite{DS10} and \cite{DSZ13}.
Here we shall describe next simplest case of $\T^2$-bundle.

\subsection{$\T^3_\theta$ as a quantum principal $\T^2$-bundle over $\T^1$}

We specify now $m=1$ and $n=2$, so the invariant subalgebra $\B$ is simply the algebra generated by $U_3$, and hence it is isomorphic to trigonometric polynomials generated by 
$z$ and $z^{-1}$.
As mentioned before $\B\hookrightarrow\A$ is a quantum principal $\T^2$-bundle in our sense
(i.e. $\A$ is a principal $H(\T^2)$-comodule algebra). 

We consider now the following spectral triple over $\A$ \cite{Ven10}. 
Let $\H_\tau$ denote the GNS Hilbert space associated to the usual trace 
$\tau$ on $\A$.
Set $\H = \H_\tau\otimes\C^2$. Next, define a Dirac operator by:
$$D = \sum_{j=1}^3 \sigma^j\delta_j,$$
where the $\sigma^j$ are the Pauli matrices. The real structure $J$ can be defined as follows: 
if $J_0$ is the Tomita-Takesaki involution on $\H_\tau$ and $c.c.$ denotes the complex 
conjugation on $\C^2$ with respect to the standard basis of $\C^2$, 
then we define $$J = J_0\otimes (i\sigma^2\circ c.c.).$$

Then $(\A,\H,D,J)$ is an odd real spectral triple, of $KR$-dimension 3, on $\A$. 
It is straightforward to check that it is $\T^2$-equivariant \cite{Ven10}. 
Moreover, the Dirac calculus $\Omega^1_D(\A)$ is easily seen to be compatible with the
 de Rham calculus on $\T^2$. Next, we have the following result.
\begin{prop}
 $(\A,\H,D,J)$ is a projectable spectral triple, with isometric fibres. Moreover, the operator $\Gamma$ can be taken equal to $\pm \sigma^3$.
\end{prop}
\begin{proof}
 Take $\Gamma = \sigma^3$ (the proof is the same for $\Gamma=-\sigma^3)$. Then $\Gamma^2 = 1$, $\Gamma^* = \Gamma$ and it commutes both with the representation of $\A$ and with the derivations $\delta_1$, $\delta_2$. Moreover, since $\sigma^2\sigma^3 = -\sigma^3\sigma^2$ and $J_0\sigma^3 = \sigma^3J_0$, we have: $J\Gamma = -\Gamma J$. Hence $\Gamma$ satisfies all the requirements of definition \ref{defTnprojodd}. It follows that $(\A,\H,D,J)$ is a projectable $\T^2$-equivariant real spectral triple.
 
 Next, according to the definitions in sect. \ref{3}, the horizontal Dirac operator is given by:
 $$D_h = {1\over 2}\Gamma[D,\Gamma]_+ = \sigma^3\delta_3.$$
 It follows that $D = D_h + D_v$, where $D_v = \sigma^1\delta_1 + \sigma^2\delta_2$; hence $(\A,\H,D,J)$ enjoys the isometric fibres property, with $Z = Z' = 0$.  
\end{proof}

Now let $\H_0$ denote the common $0$-eigenspace of $\delta_1$ and $\delta_2$. According to the results of the previous sections, we set $D_0 = D_h|_{\H_0}$ and $D_0' = \Gamma D_0$. Then the real structure $j_0$ is given by the restriction of $\Gamma J$ to $\H_0$. In particular, $j_0 = (J_0\otimes(\sigma^1\circ c.c.))|_{\H_0}$. Then $(\B,\H_0,D'_0,j_0)$ is a real spectral triple of $KR$-dimension 1.

\subsection{Twisted Dirac operators}

Now we can twist the horizontal Dirac operator $D_h$. We begin by working out a general form for strong $\T^2$-connections over $\A$.
\begin{lemma}
 Any selfadjoint strong $\T^2$-connection over $\A$, in the sense of definition \ref{sconnfamily}, is defined by two selfadjoint 1-forms $\omega_1,\omega_2\in\Omega^1_D(\A)$ of the form
 $$\omega_1 = \sigma^1 + \sigma^3\omega^3_1,$$
 $$\omega_2 = \sigma^2 + \sigma^3\omega^3_2,$$
 with $\omega_i^3 = (\omega_i^3)^*\in\B$.
\end{lemma}
\begin{proof}
 Any 1-form $\eta\in\Omega^1_D(\A)$ can be written in the following way: $\eta = \displaystyle \sum_{j=1}^3 \sigma^j\eta_j$, with $\eta_j\in\A$. Hence we write for $i=1,2$,
 $$\omega_i = \sum_{j=1}^3 \sigma^j\omega_i^j,$$
 with $\omega_i^j\in\A$. Imposing condition (i) of definition \ref{sconnfamily} we obtain that each $\omega_i^j$ has to belong to $\B$. Next noticing that each $\sigma^j$ corresponds to the (universal) 1-form $U_j^{-1}dU_j$,
the condition (ii) implies that $\omega_i^j = \delta_{ij}$ (for $i,j=1,2$). Finally, all the $\omega_i^j$ must be selfadjoint, since we are requiring the strong connection to be selfadjoint.
\end{proof}

For $k\in\Z^2$, let now $\A^{(k)}$ denote the subalgebra of $\A$ of homogeneous elements of degree $k$. Then the connection $\omega$ allows us to define a $D_0$ connection on each $\A^{(k)}$:
$$\nabla_\omega : \A^{(k)}\rightarrow\Omega^1_D(\A)\A^{(k)},$$
$$\nabla_\omega(a) = [D,a] - \sum_{i=1}^2 k_ia\omega_i.$$
By direct computation we obtain then, for any $a\in\A^{(k)}$, that:
$$\nabla_\omega(a) = \sigma^3\delta_3(a) - k_1\sigma^3a\omega_1^3 - k_2\sigma^3a\omega_2^3.$$
Before computing the twisted Dirac operator $D_\omega$, we recall the following fact: the real structure we shall use here is $\wt{j} = \Gamma j_0 = J$ (see the proof of proposition \ref{propDomegaTn}). Then, from equation \eqref{Domega}, we obtain:
$$D_\omega = \sigma^3\delta_3 - \sigma^3J\omega_1^3J^{-1}\delta_1 - \sigma^3J\omega_2^3J^{-1}\delta_2.$$
Thus if $\omega_1^3=0=\omega_2^3$, $D$ is compatible with $\omega$.
Another  ``three-dimensional" Dirac operator $\mathcal{D}_\omega$ 
can be obtained simply adding $D_v$ to $D_\omega$, it however in 
a sense corresponds to a connection possibly with torsion.

Notice that the spectral triple $(\A,\H,D_\omega)$ is not irreducible. Indeed, $\sigma^3$ commutes with $\A$ and with $D$. The reason for this is that the triple $(\B,H_0,D_0)$ is reducible: it is the direct sum of two copies (with opposite orientation) of the canonical spectral triple over the circle $S^1$. This is an expected phenomenon: indeed, in the classical (commutative case), passing from dimension 3 to dimension 1, the rank of the spinor bundle decreases and so each fibre of the spinor bundle on the total space splits as a direct sum of (two) representations of the 1-dimensional Clifford algebra.
\begin{rmk}
$D_\omega$ is the twist of the horizontal Dirac operator $D_h$; that is, of the operator $D_0$
 (given in sect. \ref{3.3}).
Instead the twist of the operator $D'_0$ can be obtained simply multiplying $D_\omega$ 
by $\Gamma$.
\end{rmk}

\section{Theta deformations}

We can combine the results of last two sections to construct a theta deformation 
of spectral triples on principal $\T^n$-bundles.
If $M$ is as in sect. \ref{6} with $m\geq 2$, then there is
a Rieffel deformation quantization $M_\theta$ \cite{R} by 'gluing a noncommutative 
torus' along the action $\alpha$ of $\T^m$, and this extends to $\alpha$-equivariant
spectral triples over $M$ \cite{cl01}.
We briefly recall  an equivalent but more `functorial'  realization \cite{CD-V} of the theta deformed
spectral triple $(C^\infty(M_{\theta}),\H_{\theta}, D_{\theta})$.
Up to a `splitting'  isomorphism 
\begin{equation}\label{splitiso}
C^\infty(M_{\theta}) \approx 
\left(
C^\infty(M)\widehat\otimes C^\infty(\Tb^2_{\theta})
\right)^{ \alpha\otimes \beta^{-1}},
\end{equation}
$$
\H_\theta := \left(L^2(M,\Sigma) \widehat\otimes L^2( \Tb^2_{\theta})\right)
^{\alpha\otimes\beta^{-1}}
$$
and  $D_{\theta}$ is the closure of $D\otimes I$,
where on the right hand sides we have 
the fixed point subalgebra or submodule of the 
action $\alpha\otimes\beta^{-1}$ of $\Tb^m$
and $\widehat\otimes$ denotes a suitable completion of $\otimes$.
Similarly the antilinear charge conjugation operator $J$ can be theta-deformed as
$$J_{\theta}=J\otimes \ast .$$
As shown in \cite{CD-V} the spectral triple $(C^\infty(S^{3}_{\theta}),\H_{\theta}, D_{\theta})$ 
together with the real structure $J_{\theta}$ 
satisfies all additional seven axioms of Connes
%\cite{C95,C96} 
required for a 'noncommutative manifold'.\\

It is also not difficult to verify that the theta deformation
(sometimes called 'twisting') behaves 'functorially' under the maps between 
manifolds
(in particular under bundle projection) and respects the properties of principal $\T^n$-bundles.
We give now more details of these affirmations.
%\footnote{
For simplicity we assume that the action of the structure group $\T^n$ coincides 
with the isometric action  on $M$ of the first factor in  $\T^n\times\T^m=\T^{n+m}$,
thus the quotient manifold $N$ still carries an isometric action of the second factor $\T^m$. 
The theta deformation $M_\theta$ constructed with the $(n+m)\times (n+m)$ matrix $\theta$ of 
parametrers  is then a principal $\T^n$-bundle over the theta deformation $N_{\theta'}$ constructed 
with the $m\times m$ matrix $\theta'$, where  $\theta'$ is the $m\times m$ 
lower--right block submatrix of $\theta$.
The right action of $\T^n$ on $M_\theta$ is just $\alpha\otimes\id$ and the bundle inclusion 
$N_\theta \hookrightarrow M_\theta$ reads  $p^*\otimes\id $, where $p^*$ is  the pullback bundle
 projection $p:M\to N$. \\

The spectral triple $(M_\theta, \H_\theta, D_\theta)$  is projectable, with 
$\Gamma_\theta= \Gamma \otimes\id $ and $(D_\theta)_h$ is just $D_{\theta'}$.
Moreover $\Omega^1_{D_\theta}(M_\theta)$ is compatible with the de Rham calculus on $\T^n$ and the requirements (a-b) in section \ref{3.4} can be seen to be satisfied.
Finally, $(\omega_\theta)_i:=\omega_i\otimes 1$ is a connection in the sense of definition \ref{sconnfamily} compatible with $D_\theta$, if $\omega_i$ is a connection compatible
with $D$.

\end{document}